\documentclass[11pt]{article}
\usepackage{amsthm,amstext}
\usepackage{amsmath}
\usepackage{graphicx}
\usepackage{amsfonts}
\usepackage{amssymb}
\theoremstyle{plain}
\newtheorem{theorem}{Theorem}[section]
\newtheorem{lemma}[theorem]{Lemma}
\newtheorem{proposition}[theorem]{Proposition}

\theoremstyle{definition}
\newtheorem{definition}[theorem]{Definition}
\newtheorem{corollary}[theorem]{Corollary}

\theoremstyle{remark}
\newtheorem{remark}{\sc Remark}

\date{}
\title{\bf Study of $\Gamma$-Semigroups via its Operator Semigroups in terms of Atanassov's Intuitionistic Fuzzy Ideals }\vspace{.25 in}
\author{ {\bf Sujit Kumar Sardar$^1$,} {\bf Samit Kumar Majumder$^2$}\\
and\\
{\bf Manasi Mandal$^3$}\\
Department of Mathematics, Jadavpur\\
University, Kolkata-700032, INDIA\\
{\tt $^1$sksardarjumath@gmail.com}\\
{\tt $^2$samitfuzzy@gmail.com}\\
{\tt $^3$manasi$_{-}$ju@yahoo.in}
 }

\begin{document}
\maketitle

\begin{abstract}

In this paper some fundamental relationships of a $\Gamma$-semigroup and its operator semigroups in terms of intuitionistic fuzzy subsets, intuitionistic fuzzy ideals, intuitionistic fuzzy prime$($semiprime$)$ ideals, intuitionistic fuzzy ideal extensions are obtained. These are then used to obtain some important characterization theorems of $\Gamma$-semigroups in terms of intuitionistic fuzzy subsets so as to highlight the role of operator semigroups in the study of $\Gamma$-semigroups in terms of intuitionistic fuzzy subsets.\\


\end{abstract}

\section{Introduction}
This is a continuation of our paper on Atanassov's Intutionistic Fuzzy Ideals of  $\Gamma$-semigroups\cite{S5} wherein we have investigated properties of Atanassov's intutionistic fuzzy\\$($prime, semiprime$)$ ideals and intutionistic fuzzy ideal extension and obtained characterization of regular $\Gamma$-semigroup and of prime ideals of $\Gamma$-semigroups.\\
\indent Dutta and Adhikari\cite{D1} have found operator semigroups of a $\Gamma$-semigroup to be a very effective tool in studying $\Gamma$-semigroups. The principal objective of this paper is to investigate as to whether the concept of operator semigroups can be made to work in the study of $\Gamma$-semigroups in terms of intuitionistic fuzzy subsets. In order to do this we deduce some fundamental relationships of a $\Gamma$-semigroup and its operator semigroups in terms of intuitionistic fuzzy subsets, intuitionistic fuzzy ideals, intuitionistic fuzzy prime$($semiprime$)$ ideals and intuitionistic fuzzy ideal extension. We then use these relationships to obtain some characterization theorems thereby establishing the effectiveness of operator semigroups in the study of $\Gamma$-semigroups in terms of intuitionistic fuzzy subsets.\\
 For preliminaries we refer to \cite{S5}.

%%%%%%%%%%%%%%%%%%%%%%%%%%%%%%%%%%%%%%%%%%%%%%%%%%%%%%%%%%%%%%%%%%%%%%%%%%%%%%%%%%%%%%%%%%
\section{Main Results}

Many results of semigroups could be extended to $\Gamma$-semigroups directly and via operator semigroups\cite{A}$($left, right$)$ of a $\Gamma$-semigroup. In this section in order to make operator semigroups of a $\Gamma$-semigroup work in the context of $IFS$\footnote{$IFS,IFS(L),IFS(R),IFS(S),IFI(S)$ respectively denote intuitionistic fuzzy subset(s), intuitionistic fuzzy subset(s) of $L,$ intuitionistic fuzzy subset(s) of $R,$ intuitionistic fuzzy subset(s) of $S,$ intuitionistic fuzzy ideal(s) of $S.$} as it worked in the study of $\Gamma$-semigroups\cite{A,D1}, we obtain various relationships between $IFI(S)$ and that of its operator semigroups. Here, among other results we obtain an inclusion preserving bijection between the set of all $IFI(S)$ and that of its operator semigroups. Among other applications of this bijection we apply it to give new proofs of its ideal analogue obtained in \cite{D1} by Dutta and Adhikari.

\begin{definition}
\cite{D1} Let $S$ be a $\Gamma$-semigroup. Let us define a relation $\rho$ on $S\times\Gamma$ as follows :$(x,\alpha)\rho(y,\beta)$ if and only if $x\alpha s=y\beta s$ for all $s\in S$ and $\gamma x\alpha=\gamma y\beta$ for all $\gamma\in\Gamma.$ Then $\rho$ is an equivalence relation. Let $[x,\alpha]$ denote the equivalence class containing $(x,\alpha)$. Let $L=\{[x,\alpha]:x\in S,\alpha\in\Gamma\}.$ Then $L$ is a semigroup with respect to the multiplication defined by $[x,\alpha][y,\beta]=[x\alpha y,\beta].$\textit{ }This semigroup $L$ is called the left operator semigroup of the $\Gamma$-semigroup $S.$ Dually the right operator semigroup $R$ of $\Gamma$-semigroup $S$ is defined where the multiplication is defined by $[\alpha,a][\beta,b]=[\alpha a\beta,b].$

If there exists an element $[e,\delta]\in L([\gamma,f]\in R)$ such that $e\delta s=s($resp. $s\gamma f=s)$ for all $s\in S$ then $[e,\delta]($resp. $[\gamma,f])$ is called the left$($resp. right$)$ unity of $S.$
\end{definition}

\begin{definition}
For an $IFS(R),$ $A=(\mu_{A},\nu_{A})$ we define an $IFS(S),$ $A^{\ast}=(\mu_{A},\nu_{A})^{\ast}=(\mu_{A}^{\ast},\nu_{A}^{\ast})$ by $\mu_{A}^{\ast}(a)=\underset{\gamma\in\Gamma}{\inf}$ $\mu_{A}([\gamma,a])$ and $\nu_{A}^{\ast}(a)=\underset{\gamma\in\Gamma}{\sup}$ $\nu_{A}([\gamma,a]),$ where $a\in S.$ For an $IFS(S),$ $B=(\mu_{B},\nu_{B})$ we define an $IFS(R),$ $B^{\ast^{\prime}}=(\mu_{B},\nu_{B})^{\ast^{'}}=(\mu_{B}^{\ast^{^{\prime}}},\nu_{B}^{\ast^{^{\prime}}})$ by $\mu_{B}^{\ast^{^{\prime}}}([\alpha,a])=\underset{s\in S}{\inf}$ $\mu_{B}(s\alpha a)$ and $\nu_{B}^{\ast^{^{\prime}}}([\alpha,a])=\underset{s\in S}{\sup}$ $\nu_{B}(s\alpha a),$ where $[\alpha,a]\in R.$ For an $IFS(L),$ $C=(\mu_{C},\nu_{C})$ we define an $IFS(S),$ $C^{+}=(\mu_{C},\nu_{C})^{+}=(\mu_{C}^{+},\nu_{C}^{+})$ by $\mu_{C}^{+}(a)=\underset{\gamma\in\Gamma}{\inf}$ $\mu_{C}([a,\gamma])$ and $\nu_{C}^{+}(a)=\underset{\gamma\in\Gamma}{\sup}$ $\nu_{C}([a,\gamma]),$ where $a\in S$. For an $IFS(S),$ $D=(\mu_{D},\nu_{D})$ we define an $IFS(L),$ $D^{+^{^{\prime}}}=(\mu_{D},\nu_{D})^{+^{^{\prime}}}=(\mu_{D}^{+^{^{\prime}}},\nu_{D}^{+^{^{\prime}}})$ by $\mu_{D}^{+^{^{\prime}}}([a,\alpha])=\underset{s\in S}{\inf}$ $\mu_{D}(a\alpha s)$ and $\nu_{D}^{+^{^{\prime}}}([a,\alpha])=\underset{s\in S}{\sup}$ $\nu_{D}(a\alpha s)$ where
$[a,\alpha]\in L.$
\end{definition}

Now we recall the following propositions from $\cite{D1}$ which were proved therein for one sided ideals. But the results can be proved to be true for two sided ideals.\newline

\begin{proposition}
$\cite{D1}$ Let $S$ be a $\Gamma$-semigroup with unities and $R$ be its right operator semigroup. If $P$ is a $LI(R)(I(R))$\footnote{$LI(R),LI(S),I(R),I(S),IFLI(R),IFLI(S),IFI(R),IFI(S)$ respectively denote left ideal(s) of $R,$ left ideal(s) of $S,$ ideal(s) of $R,$ ideal(s) of $S,$ intuitionistic fuzzy left ideal(s) of $R,$ intuitionistic fuzzy left ideal(s) of $S,$ intuitionistic fuzzy ideal(s) of $R,$ intuitionistic fuzzy ideal(s) of $S,$} then $P^{\ast}$ is a $LI(S)(I(S)).$
\end{proposition}

\begin{proposition}
$\cite{D1}$ Let $S$ be a $\Gamma$-semigroup with unities and $R$ be its right operator semigroup. If $Q$ is a $LI(S)(I(S))$ then $Q^{\ast^{^{\prime}}}$ is a $LI(R)(I(R)).$
\end{proposition}

For convenience of the readers, we may note that for a $\Gamma$-semigroup $S$ and its left, right operator semigroups $L,R$ respectively four mappings namely $()^{+},$ $()^{+^{^{\prime}}},()^{\ast},()^{\ast^{^{\prime}}}$ occur. They
are defined as follows: For $I\subseteq R,I^{\ast}=\{s\in S,[\alpha,s]\in I\forall\alpha\in\Gamma\}$; for $P\subseteq S,P^{\ast^{^{\prime}}}=\{[\alpha,x]\in R:s\alpha x\in P\forall s\in S\}$; for $J\subseteq L,J^{+}=\{s\in S,[s,\alpha]\in J\forall\alpha\in\Gamma\}$; for $Q\subseteq S,Q^{+^{^{\prime}}}=\{[x,\alpha]\in L:x\alpha s\in Q\forall s\in S\}.$

\begin{proposition}
Let $A=(\mu_{A},\nu_{A})$ be an $IFS(R),$ then $[U(\mu_{A};t)]^{\ast}=U((\mu_{A})^{\ast};t)$ and $[L(\nu_{A};t)]^{\ast}=L((\nu_{A})^{\ast};t)$ for all $t\in\lbrack0,1]$, provided the sets are non-empty.
\end{proposition}

\begin{proof}
Let $m\in S.$ Then $m \in [U(\mu_{A};t)]^{\ast}\Leftrightarrow\lbrack\gamma,m]\in U(\mu_{A};t)$\text{}$\forall\gamma\in\Gamma\Leftrightarrow\mu_{A}([\gamma,m])\geq t$\text{ }$\forall\gamma\in\Gamma
\Leftrightarrow\underset{\gamma\in\Gamma}{\inf}$ $\mu_{A}([\gamma,m])\geq t\Leftrightarrow(\mu_{A})^{\ast}(m)\geq t\Leftrightarrow m\in U((\mu_{A})^{\ast};t).$ Again let $n\in S.$ Then $n\in [L(\nu_{A};t)]^{\ast}\Leftrightarrow\lbrack\gamma,n]\in L(\nu_{A};t)$\text{}$\forall\gamma\in\Gamma\Leftrightarrow\nu_{A}([\gamma,n])\leq t$\text{ }$\forall\gamma\in\Gamma\Leftrightarrow\underset{\gamma\in\Gamma}{\sup}$ $\nu_{A}([\gamma,n])\leq t\Leftrightarrow(\nu_{A})^{\ast}(n)\leq t\Leftrightarrow n\in L((\nu_{A})^{\ast};t).$ Hence $[U(\mu_{A};t)]^{\ast}=U((\mu_{A})^{\ast};t)$ and $[L(\nu_{A};t)]^{\ast}=L((\nu_{A})^{\ast};t).$\newline
\end{proof}

\begin{proposition}
Let $B=(\mu_{B},\nu_{B})$ be an $IFS(S).$ Then $[U(\mu_{B};t)]^{\ast^{\prime}}=U((\mu_{B})^{\ast^{\prime}};t)$ and $[L(\nu_{B};t)]^{\ast^{\prime}}=L((\nu_{B})^{\ast^{\prime}};t)$ for all $t\in\lbrack0,1]$, provided the sets under consideration are non-empty.
\end{proposition}

\begin{proof}
Let $[\alpha,x]\in R$ and $t$ is as mentioned in the statement. Then $\lbrack\alpha,x]\in [U(\mu_{B};t)]^{\ast^{^{\prime}}}\Leftrightarrow m\alpha x\subseteq U(\mu_{B};t)\text{ }\forall m\in S\Leftrightarrow\mu_{B}(m\alpha x)\geq t\text{ }\forall m\in S\Leftrightarrow
\underset{m\in S}{\inf}$ $\mu_{B}(m\alpha x)\geq t\Leftrightarrow(\mu_{B})^{\ast^{^{\prime}}}([\alpha,x])\geq t\Leftrightarrow \lbrack\alpha,x]\in U((\mu_{B})^{\ast^{\prime}};t).$ Again let $[\beta,y]\in R$ and $t$ is as mentioned in the statement. Then $\lbrack\beta,y]\in [L(\nu_{B};t)]^{\ast^{^{\prime}}}\Leftrightarrow n\beta y\subseteq L(\nu_{B};t)\text{ }\forall n\in S\Leftrightarrow\nu_{B}(n\beta y)\leq t\text{ }\forall n\in S\Leftrightarrow
\underset{n\in S}{\sup}$ $\nu_{B}(n\beta y)\leq t\Leftrightarrow(\nu_{B})^{\ast^{^{\prime}}}([\beta,y])\leq t\Leftrightarrow \lbrack\beta,y]\in L((\nu_{B})^{\ast^{\prime}};t).$ Hence $[U(\mu_{B};t)]^{\ast^{\prime}}=U((\mu_{B})^{\ast^{\prime}};t)$ and $[L(\nu_{B};t)]^{\ast^{\prime}}=L((\nu_{B})^{\ast^{\prime}};t).$\newline
\end{proof}

In what follows $S$ denotes a $\Gamma$-semigroup with unities\cite{D1}, $L,$ $R$ be its left and right operator semigroups respectively.

\begin{proposition}
If $A=(\mu_{A},\nu_{A})\in IFI(R)(IFLI(R))$, then $A^{\ast}=(\mu_{A},\nu_{A})^{\ast}=(\mu_{A}^{\ast},\nu_{A}^{\ast})\in IFI(S)($respectively $IFLI(S))$.
\end{proposition}

\begin{proof}
Suppose $A=(\mu_{A},\nu_{A})\in IFI(R)$. Then $U(\mu_{A};t)$ and $L(\nu_{A};t)$ are $I(R),$ $\forall t\in [0,1].$ Hence $[U(\mu_{A};t)]^{\ast}$ and $[L(\nu_{A};t)]^{\ast}$ are $I(S),$ $\forall t\in [0,1](cf.$ Proposition $2.3).$ Now since $A=(\mu_{A},\nu_{A})$ is an $IFI(R)$, $A=(\mu_{A},\nu_{A})$ is a non-empty $IFS(R).$ Hence for some $[\alpha,m]\in
R,$ $0<\mu_{A}([\alpha,m])+\nu_{A}([\alpha,m])\leq1$. Then $U(\mu_{A};t)\neq\phi$ and $L(\nu_{A};t)\neq\phi$ where $t:=\mu_{A}([\alpha,m])=\nu_{A}([\alpha,m])$. So by
the same argument applied above $[U(\mu_{A};t)]^{\ast}\neq\phi$ and $[L(\nu_{A};t)]^{\ast}\neq\phi$. Let $u\in [U(\mu_{A};t)]^{\ast}$. Then $[\beta,u]\in U(\mu_{A};t)$ for all $\beta\in\Gamma$. Hence
$\mu_{A}([\beta,u])\geq t$. This implies that $\underset{\beta\in\Gamma}{\inf}$ $\mu_{A}([\beta,u])\geq t,$ $i.e.,$ $(\mu_{A})^{\ast}(u)\geq t$. Hence $u\in U((\mu_{A})^{\ast};t).$ Hence $U((\mu_{A})^{\ast};t)\neq\phi$. By similar argument we can show that $L((\nu_{A})^{\ast};t)\neq\phi$. Consequently, $[U(\mu_{A};t)]^{\ast}=U((\mu_{A})^{\ast};t)$ and $[L(\nu_{A};t)]^{\ast}=L((\nu_{A})^{\ast};t)(cf.$ Proposition $2.5).$ It follows that $U((\mu_{A})^{\ast};t)$ and $L((\nu_{A})^{\ast};t)$ are $I(S)$ for all $t\in [0,1].$ Hence $A^{\ast}=(\mu_{A},\nu_{A})^{\ast}=(\mu_{A}^{\ast},\nu_{A}^{\ast})$ is an $IFI(S)(cf.$ Theorem $3.10\cite{S5})$. Similarly we can prove the other case also.
\end{proof}

 In a similar fashion by using Propositions $2.4,$ $2.6$ and Theorems $3.9\cite{S5},3.10 \cite{S5}$ we deduce the following proposition.

\begin{proposition}
If $B=(\mu_{B},\nu_{B})\in IFI(S)(IFLI(S)),$ then $B^{\ast^{^{\prime}}}=(\mu_{B},\nu_{B})^{\ast^{'}}=(\mu_{B}^{\ast^{^{\prime}}},\nu_{B}^{\ast^{^{\prime}}})\in IFI(R)($respectively $IFLI(R))$.
\end{proposition}

\begin{remark}
The left operator analogues of Propositions $2.3$-$2.8.$ are also true.
\end{remark}

In view of Remark $1,$ we deduce the following theorem.

\begin{theorem}
Let $S$ be a $\Gamma$-semigroup with unities and $L$ be its left operator semigroup. Then there exists an inclusion preserving bijection $A \mapsto A^{+^{^{\prime}}}$ between the set of all $IFI(S)(IFRI(S))$\footnote{$IFRI(L),IFRI(S)$ respectively denote intuitionistic fuzzy right ideal(s) of $L$ and intuitionistic fuzzy right ideal(s) of $S.$} and set of all $IFI(L)($resp. $IFRI(L)),$ where $A=(\mu_{A},\nu_{A})$ is an $IFI(S)($resp. $IFRI(S)).$
\end{theorem}

\begin{proof}
Let $A=(\mu_{A},\nu_{A})\in IFI(S)(IFRI(S))$ and $x\in S.$ Then%
\[
(\mu_{A}^{+^{^{\prime}}})^{+}(x)=\underset{\gamma\in\Gamma}{\inf}\mu_{A}^{+^{^{\prime}}}([x,\gamma])=\underset{\gamma\in\Gamma}{\inf} \underset{s\in S}{[\inf}\mu_{A}(x\gamma s)]\geq\mu_{A}(x).
\] \\ Again%
\[
(\nu_{A}^{+^{^{\prime}}})^{+}(x)=\underset{\gamma\in\Gamma}{\sup}%
\nu_{A}^{+^{^{\prime}}}([x,\gamma])=\underset{\gamma\in\Gamma}{\sup}%
\underset{s\in S}{[\sup}\nu_{A}(x\gamma s)]\leq\nu_{A}(x).
\]

Hence $A\subseteq(A^{+^{^{\prime}}})^{+}.$ Let $[\gamma,f]$ be the
right unity of $S.$ Then $x\gamma f=x$ for all $x\in S.$ Then%
\[
\mu_{A}(x)=\mu_{A}(x\gamma f)\geq\underset{\alpha\in\Gamma}{\inf}\underset{s\in
S}{[\inf}\mu_{A}(x\alpha s)]=\underset{\alpha\in\Gamma}{\inf}\mu_{A}
^{+^{^{\prime}}}([x,\alpha])=(\mu_{A}^{+^{^{\prime}}})^{+}(x).
\]\\ Again%
\[
\nu_{A}(x)=\nu_{A}(x\gamma f)\leq\underset{\alpha\in\Gamma}{\sup}\underset{s\in
S}{[\sup}\nu_{A}(x\alpha s)]=\underset{\alpha\in\Gamma}{\sup}\nu_{A}
^{+^{^{\prime}}}([x,\alpha])=(\nu_{A}^{+^{^{\prime}}})^{+}(x).
\]

So $A\supseteq(A^{+^{^{\prime}}})^{+}.$ Hence $(A^{+^{^{\prime}}})^{+}=A.$ Thus the said mapping is one-one.
 Now let $B=(\mu_{B},\nu_{B})\in IFI(L)(IFRI(S)).$ Then%
\begin{align*}
(\mu_{B}^{+})^{+^{^{\prime}}}([x,\alpha])  &  =\underset{s\in S}{\inf}\mu_{B}
^{+}(x\alpha s)=\underset{s\in S}{\inf}[\underset{\gamma\in\Gamma}{\inf}%
\mu_{B}([x\alpha s,\gamma])]\\
&  =\underset{s\in S}{\inf}[\underset{\gamma\in\Gamma}{\inf}\mu_{B}([x,\alpha
][s,\gamma])]\geq\mu_{B}([x,\alpha]).
\end{align*}
Again%
\begin{align*}
(\nu_{B}^{+})^{+^{^{\prime}}}([x,\alpha])  &  =\underset{s\in S}{\sup}\nu_{B}
^{+}(x\alpha s)=\underset{s\in S}{\sup}[\underset{\gamma\in\Gamma}{\sup}%
\nu_{B}([x\alpha s,\gamma])]\\
&  =\underset{s\in S}{\sup}[\underset{\gamma\in\Gamma}{\sup}\nu_{B}([x,\alpha
][s,\gamma])]\geq\nu_{B}([x,\alpha]).
\end{align*}
So $B\subseteq(B^{+})^{+^{^{\prime}}}.$ Let $[e,\delta]$ be the left unity
of $L.$ Then%
\begin{align*}
\mu_{B}([x,\alpha])  &  =\mu_{B}([x,\alpha][e,\delta])\geq\underset{s\in S}{\inf
}[\underset{\gamma\in\Gamma}{\inf}\mu_{B}([x,\alpha][s,\gamma])]\\
&  =(\mu_{B}^{+})^{+^{^{\prime}}}([x,\alpha]).
\end{align*}
Again%
\begin{align*}
\nu_{B}([x,\alpha])  &  =\nu_{B}([x,\alpha][e,\delta])\leq\underset{s\in S}{\sup
}[\underset{\gamma\in\Gamma}{\sup}\nu_{B}([x,\alpha][s,\gamma])]\\
&  =(\nu_{B}^{+})^{+^{^{\prime}}}([x,\alpha])
\end{align*}
So $B\supseteq(B^{+})^{+^{^{\prime}}}$ and hence $B=(B^{+}%
)^{+^{^{\prime}}}.$ Consequently, the correspondence $A\mapsto A^{+^{^{\prime
}}}$ is a bijection. Now let $C=(\mu_{C},\nu_{C}),D=(\mu_{D},\nu_{D})\in IFI(S)(IFRI(S))$ be such
that $C\subseteq D,$ $i.e.,\mu_{C}\subseteq\mu_{D}$ and $\nu_{C}\supseteq\nu_{D}.$ Then for all $[x,\alpha]\in L,$%
\[
\mu_{C}^{+^{^{\prime}}}([x,\alpha])=\underset{s\in S}{\inf}\mu
_{C}(x\alpha s)\leq\underset{s\in S}{\inf}\mu_{D}(x\alpha s)=\mu
_{D}^{+^{^{\prime}}}([x,\alpha])
\] and%
\[
\nu_{C}^{+^{^{\prime}}}([x,\alpha])=\underset{s\in S}{\sup}\nu
_{C}(x\alpha s)\geq\underset{s\in S}{\sup}\nu_{D}(x\alpha s)=\nu
_{D}^{+^{^{\prime}}}([x,\alpha]).
\]

Thus $\mu_{C}^{+^{^{\prime}}}\subseteq\mu_{D}^{+^{^{\prime}}}$ and $\nu_{C}^{+^{^{\prime}}}\supseteq\nu_{D}^{+^{^{\prime}}}.$ Consequently, $C^{+^{^{\prime}}}\subseteq D^{+^{^{\prime}}}.$
Hence $A\mapsto A^{+^{^{\prime}}}$ is an inclusion preserving bijection.
The rest of the proof follows from Remark $1.$\newline
\end{proof}

In a similar way by using Proposition $2.7$ and Proposition $2.8$ we can deduce the following theorem.

\begin{theorem}
Let $S$ be a $\Gamma$-semigroup with unities and $R$ be its right operator semigroup. Then there exists an inclusion preserving bijection $B\mapsto B^{\ast^{^{\prime}}}$ between the set of all $IFI(S)(IFLI(S))$ and set of all $IFI(R)($resp. $IFLI(R)),$ where $B=(\mu_{B},\nu_{B})$ is an $IFI($resp. $IFLI(S)).$
\end{theorem}

Now to apply the above theorem for giving a new proof of Theorem
$4.6\cite{D1}$ and its two sided ideal analogue we deduce the
following lemmas.

\begin{lemma}
Let $I$ be a $LI(R)(I(R))$ of a $\Gamma$-semigroup $S$ and $P=(\chi_{I},\chi^{c}_{I})$ where $\chi_{I}$ is the characteristic function of $I.$ Then $P^{\ast}=(\chi_{I},\chi^{c}_{I})^{\ast}=((\chi_{I})^{\ast},(\chi^{c}_{I})^{\ast})=(\chi_{I^{\ast}},\chi^{c}_{I^{\ast}}).$
\end{lemma}

\begin{proof}
Suppose $s\in I^{\ast}$. Then $[\beta,s]\in I$ for all $\beta\in\Gamma$. This
means $\underset{\beta\in\Gamma}{\inf}(\chi_{I}([\beta,s]))=1$ and $\underset{\beta\in\Gamma}{\sup}(\chi^{c}_{I}([\beta,s]))=0$. Also
$\chi_{I^{\ast}}(s)=1$ and $\chi^{c}_{I^{\ast}}(s)=0$. Now suppose $s\notin I^{\ast}$. Then there exists
$\delta\in\Gamma$ such that $[\delta,s]\notin I.$ Hence $\chi_{I}([\delta,s])=0,\chi^{c}_{I}([\delta,s])=1$ and so $\underset{\beta\in\Gamma}{\inf}(\chi_{I}([\beta,s]))=0,\underset{\beta\in\Gamma}{\sup}(\chi^{c}_{I}([\beta,s]))=1$. Hence $(\chi_{I})^{\ast}(s)=0$ and $(\chi^{c}_{I})^{\ast}(s)=1.$ Again $(\chi_{I^{\ast}})(s)=0$ and $(\chi^{c}_{I^{\ast}})(s)=1.$ Thus $P^{\ast}=(\chi_{I},\chi^{c}_{I})^{\ast}=((\chi_{I})^{\ast},(\chi^{c}_{I})^{\ast})=(\chi_{I^{\ast}},\chi^{c}_{I^{\ast}}).$\newline
\end{proof}

The following lemma follows in a similar way.

\begin{lemma}
Let $I$ be a $RI(S)(I(S)),$ $P=(\chi_{I},\chi^{c}_{I})$ and $R$ be the
right operator semigroup of $S$. Then $P^{\ast^{\prime}}=(\chi_{I},\chi^{c}_{I})^{\ast^{'}}=((\chi_{I})^{\ast^{\prime}},(\chi^{c}_{I})^{\ast^{\prime}})=(\chi_{I^{\ast^{\prime}}},\chi^{c}_{I^{\ast^{\prime}}}),$ where $\chi_{I}$ is the characteristic function of $I$.
\end{lemma}

\begin{remark}
By drawing an analogy we deduce results similar to the above lemmas for left operator semigroup $L$ of the $\Gamma$-semigroup $S,i.e.,$ for the functions $+$ and $+^{\prime}$ .
\end{remark}

Now we present a new proof of the following result which is originally due to Dutta and Adhikari\cite{D1}.

\begin{theorem}
$\cite{D1}$Let S be a $\Gamma$-semigroup with unities. Then there
exists an inclusion preserving bijection between the set of all $I(S)(LI(S))$ and that of its right operator semigroup $R$ via the mapping $I\rightarrow I^{{\ast}^{^{\prime}}}$.
\end{theorem}

\begin{proof}
Let us denote the mapping $I\rightarrow I^{{\ast}^{^{\prime}}}$ by $\phi$.
This is actually a mapping follows from Proposition $2.8$. Now let
$\phi(I_{1})=\phi(I_{2})$. Then $I_{1}^{{\ast}^{^{\prime}}}=I_{2}^{{\ast
}^{^{\prime}}}.$ This implies that $(\chi_{I_{1}^{\ast^{\prime}}},\chi^{c}_{I_{1}^{\ast^{\prime}}})=(\chi_{I_{2}^{\ast^{\prime}}},\chi^{c}_{I_{2}^{\ast^{\prime}}})($where $\chi_{I}$ is
the characteristic function $I).$ Hence by Lemma $2.12,(\chi_{I_{1}%
},\chi^{c}_{I_{1}%
})^{{\ast}^{^{\prime}}}=(\chi_{I_{2}},\chi^{c}_{I_{2}})^{{\ast}^{^{\prime}}}$. This together
with Theorem $2.10,$ gives $(\chi_{I_{1}},\chi^{c}_{I_{1}})=(\chi_{I_{2}},\chi^{c}_{I_{2}})$ whence
$I_{1}=I_{2}$. Consequently $\phi$ is one-one. Let $I$ be a $I(R)(LI(R)).$ Then $(\chi_{I},\chi^{c}_{I})$ is an $IFI(R)(IFLI(R)).$ Hence by Theorem $2.10,$ $((\chi_{I},\chi^{c}_{I})^{\ast})^{{\ast
}^{^{\prime}}}=(\chi_{I},\chi^{c}_{I})$. This implies that $(\chi_{(I^{\ast})^{{\ast
}^{^{\prime}}}},\chi^{c}_{(I^{\ast})^{{\ast
}^{^{\prime}}}})=(\chi_{I},\chi^{c}_{I})$ $(cf$. Lemma $2.11$ and Lemma $2.12).$ Hence
$(I^{\ast})^{{\ast}^{^{\prime}}}=I,i.e.,$ $\phi(I^{\ast})=I$. Now since
$I^{\ast}$ is a $I(S)(LI(S))(cf$. Proposition $2.3),$ it
follows that $\phi$ is onto. Let $I_{1},I_{2}$ be two $I(S)(LI(S))$
with $I_{1}\subseteq I_{2}.$ Then $\chi_{I_{1}}\subseteq\chi_{I_{2}}$ and $\chi^{c}_{I_{1}}\supseteq\chi^{c}_{I_{2}}$.
Hence by Theorem $2.10,$ we see that $(\chi_{I_{1}})^{{\ast}^{^{\prime}}%
}\subseteq(\chi_{I_{2}})^{{\ast}^{^{\prime}}}$ and $(\chi^{c}_{I_{1}})^{{\ast}^{^{\prime}}%
}\supseteq(\chi^{c}_{I_{2}})^{{\ast}^{^{\prime}}}$ $i.e.,$ $\chi
_{I_{1}^{{\ast}^{^{\prime}}}}\subseteq\chi_{I_{2}^{{\ast}^{^{\prime}}}}$ and $\chi^{c}
_{I_{1}^{{\ast}^{^{\prime}}}}\supseteq\chi^{c}_{I_{2}^{{\ast}^{^{\prime}}}}$
$(cf$. Lemma $2.12)$ which gives $I_{1}^{{\ast}^{^{\prime}}}\subseteq
I_{2}^{{\ast}^{^{\prime}}}.$\newline
\end{proof}

\begin{remark}
Now by using a similar argument as above and with the help of lemmas dual to the lemmas $2.11,2.12(cf.$ Remark $2)$ and Theorem $2.9$ we deduce that the mapping $()^{+^{^{\prime}}}$ is an inclusion preserving bijection$($with $()^{+}$ as the inverse$)$ between the set of all $I(S)(RI(S))$ and that of its left operator semigroup $L.$
\end{remark}

In what follows $S$ denotes a $\Gamma$-semigroup not necessarily with unities, $L,$ $R$ be its left and right operator semigroups respectively.

\begin{proposition}
$\cite{D1,S3}$ Let $S$ be a $\Gamma$-semigroup and $R$ be its right operator semigroup. If $P$ is $PI(R)(SPI(R))$\footnote{$PI(R),PI(S),SPI(R),SPI(S),IFPI(R),IFPI(S),IFSPI(R),IFSPI(S)$ respectively denote prime ideal(s) of $R,$ prime ideal(s) of $S,$ semiprime ideal(s) of $R,$ semiprime ideal(s) of $S,$ intuitionistic fuzzy prime ideal(s) of $R,$ intuitionistic fuzzy prime ideal(s) of $S,$ intuitionistic fuzzy semiprime ideal(s) of $R,$ intuitionistic fuzzy semiprime ideal(s) of $S.$} then $P^{\ast}$ is $PI(S)(SPI(S)).$
\end{proposition}

\begin{proposition}
$\cite{D1,S3}$ Let $S$ be a $\Gamma$-semigroup and $R$ be its right operator semigroup. If $Q$ is $PI(S)(SPI(S))$ then $Q^{\ast^{^{\prime}}}$ is $PI(R)(SPI(R)).$
\end{proposition}

\begin{proposition}
If $A=(\mu_{A},\nu_{A})\in IFPI(R)(IFSPI(R))$, then $A^{\ast}=(\mu_{A}^{\ast},\nu_{A}^{\ast})\in IFPI(S)($resp. $IFSPI(S))$.
\end{proposition}

\begin{proof}
Let $A=(\mu_{A},\nu_{A})\in IFPI(R).$ Then it is in $IFI(R).$ Hence by Proposition $2.7,$ $A^{\ast}=(\mu_{A}^{\ast},\nu_{A}^{\ast})\in IFI(S).$ Since $A=(\mu_{A},\nu_{A})\in IFPI(R),$ so $U(\mu_{A};t)$ and $L(\mu_{A};t)$ are $PI(R).$ Now by Proposition $2.14,$ for all $t\in [0,1],$ $[U(\mu_{A};t)]^{\ast}$ and $[L(\nu_{A};t)]^{\ast}$ are $PI(S).$ By Proposition $2.5,$ $[U(\mu_{A};t)]^{\ast}=U((\mu_{A})^{\ast};t)$ and $[L(\nu_{A};t)]^{\ast}=L((\nu_{A})^{\ast};t).$ So $U((\mu_{A})^{\ast};t)$ and $L((\nu_{A})^{\ast};t)$ are $PI(S).$ Hence $A^{\ast}=(\mu^{\ast}_{A},\nu^{\ast}_{A})\in IFPI(S).$ Similarly we can prove the other case also.
\end{proof}

In a similar fashion by using Propositions $2.6,2.8$ and Theorem $3.9\cite{S5},3.10\cite{S5}$ we deduce the following proposition.

\begin{proposition}
If $B=(\mu_{B},\nu_{B})\in IFPI(S)(IFSPI(S)),$ then $B^{\ast^{^{\prime}}}=(\mu_{B}^{\ast^{^{\prime}}},\nu_{B}^{\ast^{^{\prime}}})\in IFPI(R).$
\end{proposition}

\begin{remark}
We can also deduce the following left operator analogues of Propositions $2.14$-$2.17$.
\end{remark}

The following theorem is on the inclusion preserving bijection between the set of all $IFPI(S)$ and the set of all $IFPI(R).$ It may be noted that $S$ need not have unities here which was the case for the set of all $IFI(cf.$ Theorems $2.9, 2.10).$

\begin{theorem}
Let $S$ be a $\Gamma$-semigroup and $R$ be its right operator semigroup. Then there exist an inclusion preserving bijection $B=(\mu_{B},\nu_{B})\mapsto B^{\ast^{^{\prime}}}=(\mu^{\ast^{'}}_{B},\nu^{\ast^{'}}_{B})$ between the set of all $IFPI(S)(IFSPI(S))$ and set of all $IFPI(R)($resp. $IFSPI(R)).$
\end{theorem}

\begin{proof}
Let $B=(\mu_{B},\nu_{B})\in IFPI(R)$ and $x\in S.$ Then
$$
\begin{array}{ll}
(\mu^{\ast^{'}}_{B})^{\ast}(x)&=\underset{\gamma\in\Gamma}{\inf}\mu^{\ast^{'}}_{B}([\gamma,x])=\underset{\gamma\in\Gamma}{\inf}\underset{s\in S}{\inf}\mu_{B}(s\gamma x)\geq\mu_{B}(x)(\text{since }B\in IFI(S)).
\end{array}
$$
$$
\begin{array}{ll}
\text{Again for }x\in S, (\mu^{\ast^{'}}_{B})^{\ast}(x)&=\underset{\gamma\in\Gamma}{\inf}\underset{s\in S}{\inf}\mu_{B}(s\gamma x)=\underset{s\in S}{\inf}\underset{\gamma\in\Gamma}{\inf}\mu_{B}(s\gamma x)\\
&=\underset{s\in S}{\inf}\max\{\mu_{B}(s),\mu_{B}(x)\}(\text{since }B\in IFPI(S))\\
&\leq\max\{\mu_{B}(x),\mu_{B}(x)\}=\mu_{B}(x).
\end{array}
$$
Hence $(\mu^{\ast^{'}}_{B})^{\ast}(x)=\mu_{B}(x).$ Also
$$
\begin{array}{ll}
(\nu^{\ast^{'}}_{B})^{\ast}(x)&=\underset{\gamma\in\Gamma}{\sup}\nu^{\ast^{'}}_{B}([\gamma,x])=\underset{\gamma\in\Gamma}{\sup}\underset{s\in S}{\sup}\mu_{B}(s\gamma x)\leq\nu_{B}(x)(\text{since }B\in IFI(S)).
\end{array}
$$
$$
\begin{array}{ll}
\text{Again for }x\in S, (\nu^{\ast^{'}}_{B})^{\ast}(x)&=\underset{\gamma\in\Gamma}{\sup}\underset{s\in S}{\sup}\nu_{B}(s\gamma x)=\underset{s\in S}{\sup}\underset{\gamma\in\Gamma}{\sup}\nu_{B}(s\gamma x)\\
&=\underset{s\in S}{\sup}\min\{\nu_{B}(s),\nu_{B}(x)\}(\text{since }B\in IFPI(S))\\
&\geq\min\{\nu_{B}(x),\nu_{B}(x)\}=\nu_{B}(x).
\end{array}
$$
Hence $(\nu^{\ast^{'}}_{B})^{\ast}(x)=\nu_{B}(x).$ Consequently, $(B^{\ast^{'}})^{\ast}=B.$ Hence the mapping is one-one. Now let $[\alpha,x]\in R.$ Then
$$
\begin{array}{ll}
(\mu^{\ast}_{B})^{\ast^{'}}([\alpha,x])&=\underset{s\in S}{\inf}\mu^{\ast}_{B}(s\alpha x)=\underset{s\in S}{\inf}\underset{\beta\in\Gamma}{\inf}\mu_{B}([\beta,s\alpha x])\\
&=\underset{s\in S}{\inf}\underset{\beta\in\Gamma}{\inf}\mu_{B}([\beta,s][\alpha,x])\geq\mu_{B}([\alpha,x])
\ \ \ \ \ \ \ \ \ \ \ \ ......................(*_{1}).
\end{array}
$$
$$
\begin{array}{ll}
\text{Also, }(\nu^{\ast}_{B})^{\ast^{'}}([\alpha,x])&=\underset{s\in S}{\sup}\nu^{\ast}_{B}(s\alpha x)=\underset{s\in S}{\sup}\underset{\beta\in\Gamma}{\sup}\nu_{B}([\beta,s\alpha x])\\
&=\underset{s\in S}{\sup}\underset{\beta\in\Gamma}{\sup}\nu_{B}([\beta,s][\alpha,x])\leq\nu_{B}([\alpha,x])
\ \ \ \ \ ......................(**_{1}).
\end{array}
$$
Since $B=(\mu_{B},\nu_{B})\in IFPI(R),$ for all $ \alpha, \beta \in \Gamma,$ for all $x,s \in S.$ $\mu_{B}([\alpha,x][\beta,s])=\max\{\mu_{B}([\alpha,x]),\mu_{B}([\beta,s])\}$ and $\nu_{B}([\alpha,x][\beta,s])=\min\{\nu_{B}([\alpha,x]),\nu_{B}([\beta,s])\}\forall s\in S,\forall\beta\\\in\Gamma .$ Hence for $s=x$ and $\beta=\alpha$ we obtain $\mu_{B}([\alpha,x][\beta,s])=\mu_{B}([\alpha,x])$ and $\nu_{B}([\alpha,x][\beta,s])\\=\nu_{B}([\alpha,x]).$ This together with the relations $(\mu^{\ast}_{B})^{\ast^{'}}([\alpha,x])=\underset{s\in S}{\inf}\underset{\beta\in\Gamma}{\inf}\mu_{B}([\alpha,x][\beta,s])$ and $(\nu^{\ast}_{B})^{\ast^{'}}([\alpha,x])=\underset{s\in S}{\sup}\underset{\beta\in\Gamma}{\sup}\nu_{B}([\alpha,x][\beta,s])$ give $(\mu^{\ast}_{B})^{\ast^{'}}([\alpha,x])\leq\mu_{B}([\alpha,x])..........(*_{2})$ and $(\nu^{\ast}_{B})^{\ast^{'}}([\alpha,x])\geq\nu_{B}([\alpha,x])..........(**_{2}).$ By $(*_{1})$ and $(*_{2})$ we obtain $(\mu^{\ast}_{B})^{\ast^{'}}([\alpha,x])=\mu_{B}([\alpha,x])$ and by $(**_{1})$ and $(**_{2})$ we have $(\nu^{\ast}_{B})^{\ast^{'}}([\alpha,x])=\nu_{B}([\alpha,x]).$ Consequently, $(B^{\ast})^{\ast^{'}}=B.$ Hence the mapping is onto. Inclusion preserving property is similar as in Theorem $2.9.$ Hence $B\mapsto B^{\ast^{^{\prime}}}$ is an inclusion preserving bijection.
\end{proof}

\begin{remark}
$(i)$ Similar results hold for $IFSPI(S).$ $(ii)$ Similar result holds for the $\Gamma$-semigroup $S$ and the left operator semigroup $L$ of $S.$
\end{remark}

\begin{corollary}
Let $S$ be a $\Gamma$-semigroup and $R,L$ be respectively its right and left operator semigroups. Then there exists an inclusion preserving bijection between the set of all $IFPI(R)(IFSPI(R))$ and the set of all $IFPI(L)(IFSPI(L)).$
\end{corollary}

\begin{remark}
In view of Theorem $2.18,$ we see that in a $\Gamma$-semigroup $S$ with unities the above result also holds for $IFI.$
\end{remark}

Now we revisit the following theorem which is originally due to Dutta and Adhikari\cite{D1} via intutionistic
fuzzy ideals by using Theorem $2.18$ and applying  similar argument as applied in Theorem $2.13.$

\begin{theorem}
Let $S$ be a $\Gamma$-semigroup. Then there exists an inclusion preserving bijection between the set of all $PI(S)(SPI(S))$ and that of its right operator semigroup $R$ via the mapping $I\rightarrow I^{\ast^{'}}.$
\end{theorem}

The definition of an $IFE,$\footnote{$IFE$ denote an intuitionsitic fuzzy extension.} $<x,A>$ of an $IFS,$ $A=(\mu_{A},\nu_{A})$ is given in \cite{S5}. Now by routine verification we obtain the following two propositions.

\begin{proposition}
Let $S$ be a commutative $\Gamma$-semigroup and $L(R)$ the left$($respectively the right$)$ operator semigroups of $S.$ Let $A=(\mu_{A},\nu_{A})$ be an $IFLI(S)(IFRI(S),IFI(S))$ then $<x,A^{+^{'}}>($respectively $<x,A^{\ast^{'}}>)$ is an $IFLI(L(R))(IFRI(L(R)),IFI(L(R)))$ for all $x\in L(R).$
\end{proposition}

\begin{proposition} $($With same notation as in the above proposition$)$ If $B=(\mu_{B},\nu_{B})$ is an $IFLI(L(R))(IFRI(L(R)),IFI(L(R)))$ then $<x,B^{+}>($respectively $<x,B^{\ast}>)$ is an $IFLI(S)(IFRI(S),IFI(S))$ for all $x\in S.$
\end{proposition}

Now we deduce the following two lemmas on the relationships between a $\Gamma$-semigroup and its operator semigroups in terms of $IFE.$

\begin{lemma}
Let $A=(\mu_{A},\nu_{A})$ be an $IFS(S)$ where $S$ is commutative. Then for all $x\in S$,

$(1)$ $<x,A>^{\ast^{'}}\subseteq <[\alpha,x],A^{\ast^{'}}>\forall\alpha\in\Gamma.$\\

$(2)$ $<x,A>^{\ast^{'}}=(<x,\mu_{A}>^{\ast^{'}},<x,\nu_{A}>^{\ast^{'}})$

$\ \ \ \ \ \ \ \ \ \ \ \ \ \ \ \ \ =(\underset{\alpha\in\Gamma}{\inf}<[\alpha,x],\mu_{A}^{\ast^{'}}>,\underset{\alpha\in\Gamma}{\sup}<[\alpha,x],\nu_{A}^{\ast^{'}}>).$
\end{lemma}

\begin{proof}
$(1)$ Let $[\beta,y]\in R.$ Then $<x,\mu_{A}>^{\ast^{'}}([\beta,y])=\underset{s\in S}{\inf}$ $<x,\mu_{A}>(s\beta y)=\underset{s\in S}{\inf}\underset{\gamma\in\Gamma}{\inf}$ $\mu_{A}(x\gamma s\beta y)=\underset{\gamma\in\Gamma}{\inf}\underset{s\in S}{\inf} $ $\mu_{A}(x\gamma s\beta y).$ Again $<[\alpha,x],\mu^{\ast^{'}}_{A}>([\beta,y])=\mu^{\ast^{'}}_{A}([\alpha,x][\beta,y])=\mu^{\ast^{'}}_{A}([\alpha,x\beta y])=\underset{s\in S}{\inf}\mu^{\ast}_{A}(s\alpha x\beta y)=\underset{s\in S}{\inf}\mu^{\ast}_{A}(x\alpha s\beta y)($using the commutativity of $S).$ Since $\underset{\gamma\in\Gamma}{\inf}\underset{s\in S}{\inf}\mu^{\ast}_{A}(x\gamma s\beta y)\leq\underset{s\in S}{\inf}\mu^{\ast}_{A}(x\alpha s\beta y),$ we obtain $<x,\mu_{A}>^{\ast^{'}}([\beta,y])\leq<[\alpha,x],\mu^{\ast^{'}}_{A}>([\beta,y]).$ By similar  argument we can show that $<x,\nu_{A}>^{\ast^{'}}([\beta,y])\geq <[\alpha,x],\nu^{\ast^{'}}_{A}>([\beta,y]).$ Hence $<x,A>^{\ast^{'}}\subseteq <[\alpha,x],A^{\ast^{'}}>\forall\alpha\in\Gamma.$\\

$(2)$ Let $[\beta,y]\in R.$ Then $\underset{\alpha\in\Gamma}{\inf}$ $<[\alpha,x],\mu^{\ast^{'}}_{A}>([\beta,y])=\underset{\alpha\in\Gamma}{\inf}$ $\mu^{\ast^{'}}_{A}([\alpha,x][\beta,y])=\underset{\alpha\in\Gamma}{\inf}$ $\mu^{\ast^{'}}_{A}([\alpha,x\beta y])=\underset{\alpha\in\Gamma}{\inf}\underset{s\in S}{\inf}$ $\mu^{\ast}_{A}(s\alpha x\beta y)=\underset{s\in S}{\inf}$ $<x,\mu_{A}>(s\beta y)=<x,\mu_{A}>^{\ast^{'}}([\beta,y]).$ By applying similar argument we obtain $\underset{\alpha\in\Gamma}{\sup}$ $<[\alpha,x],\nu^{\ast^{'}}_{A}>([\beta,y])=<x,\nu_{A}>^{\ast^{'}}([\beta,y]).$ Hence  $<x,A>^{\ast^{'}}=(<x,\mu_{A}>^{\ast^{'}},<x,\nu_{A}>^{\ast^{'}})=(\underset{\alpha\in\Gamma}{\inf}<[\alpha,x],\mu_{A}^{\ast^{'}}>,\underset{\alpha\in\Gamma}{\sup}<[\alpha,x],\nu_{A}^{\ast^{'}}>).$
\end{proof}

\begin{lemma}
If $B=(\mu_{B},\nu_{B})$ is an $IFS(R)$ then ~for all~ $x\in S,$ $<[\beta,x],B>^{\ast}\supseteq  <x,B^{\ast}>\forall\beta\in\Gamma.$
\end{lemma}

\begin{proof}
Let $p\in S.$ Then $<[\beta,x],\mu_{B}>^{\ast}(p)=\underset{\gamma\in\Gamma}{\inf}$ $<[\beta,x],\mu_{B}>([\gamma,p])=\underset{\gamma\in\Gamma}{\inf}$ $\mu_{B}([\beta,x][\gamma,p])=\underset{\gamma\in\Gamma}{\inf}$ $\mu_{B}([\beta,x\gamma p]).$ Again $<x,\mu^{\ast}_{B}>(p)=\underset{\gamma\in\Gamma}{\inf}$ $\mu^{\ast}_{B}(x\gamma p)=\underset{\gamma\in\Gamma}{\inf}\underset{\beta\in\Gamma}{\inf}$ $\mu_{B}([\beta,x\gamma p])=\underset{\beta\in\Gamma}{\inf}\underset{\gamma\in\Gamma}{\inf}$ $\mu_{B}([\beta,x\gamma p]).$ Since $\underset{\gamma\in\Gamma}{\inf}$ $\mu_{B}([\beta,x\gamma p])\geq\underset{\beta\in\Gamma}{\inf}\underset{\gamma\in\Gamma}{\inf}$ $\mu_{B}([\beta,x\gamma p]),$ we have $<[\beta,x],\mu_{B}>^{\ast}(p)\geq <x,\mu^{\ast}_{B}>(p).$ Similarly we can show that $<[\beta,x],\nu_{B}>^{\ast}(p)\leq <x,\nu^{\ast}_{B}>(p).$ Hence $<[\beta,x],B>^{\ast}\supseteq <x,B^{\ast}>\forall\beta\in\Gamma.$
\end{proof}

\begin{lemma}
Let $I$ be an $I(S)$. Then $((\mu_{I})^{\ast^{'}},(\nu_{I})^{\ast^{'}})=(\mu_{I^{\ast^{'}}},\nu_{I^{\ast^{'}}}).$
\end{lemma}

\begin{proof}
Let $[\beta,y]\in R.$ Then $(\mu_{I})^{\ast^{'}}([\beta,y])=\underset{s\in S}{\inf}$ $\mu_{I}(s\beta y)$ and $(\nu_{I})^{\ast^{'}}([\beta,y])=\underset{s\in S}{\sup}$ $\nu_{I}(s\beta y).$ Suppose $[\beta,y]\in I^{\ast^{'}}.$ Then $s\beta y\in I$ for all $s\in S.$ Hence $\mu_{I}(s\beta y)=1$ and $\nu_{I}(s\beta y)=0$ for all $s\in S$ which implies that $\underset{s\in S}{\inf}$ $\mu_{I}(s\beta y)=1$ and $\underset{s\in S}{\sup}$ $\nu_{I}(s\beta y)=0$ whence $(\mu_{I})^{\ast^{'}}([\beta,y])=1$ and $(\nu_{I})^{\ast^{'}}([\beta,y])=0.$ Also $\mu_{I^{\ast^{'}}}([\beta,y])=1$ and $\nu_{I^{\ast^{'}}}([\beta,y])=0.$ If $[\beta,y]\notin I^{\ast^{'}}$ then $(\mu_{I^{\ast^{'}}})([\beta,y])=0$, $(\nu_{I^{\ast^{'}}})([\beta,y])=1$ and there exists $s  \in S$ such that $s\beta y\notin I $. Hence $\mu_{I}(s\beta y)= 0$ and $\nu_{I}(s\beta y)= 1$ whence $(\mu_{I})^{\ast^{'}}([\beta,y])=0$ and $(\nu_{I})^{\ast^{'}}([\beta,y])=1.$ Consequently, $((\mu_{I})^{\ast^{'}},(\nu_{I})^{\ast^{'}})=(\mu_{I^{\ast^{'}}},\nu_{I^{\ast^{'}}}).$
\end{proof}

\begin{lemma}
$\cite{D3}$ Let $\{A_{\alpha}\}_{\alpha\in I}$ be a family of $I(S).$ Then $(\underset{\alpha\in I}{\bigcap} A_{\alpha})^{\ast^{'}}=\underset{\alpha\in I}{\bigcap}(A_{\alpha})^{\ast^{'}}.$
\end{lemma}

\begin{lemma}
Let $S$ be a $\Gamma$-semigroup, $R$ be its right operator semigroup and $\{A_{i}\}_{i\in I}=(\mu_{A_{i}},\nu_{A_{i}})_{i\in I}$ be a family of $IFS(S)$ such that $A=(\mu_{A},\nu_{A}):=\underset{i\in I}{\inf}$ $ A_{i}=(\underset{i\in I}{\inf}$ $\mu_{A_{i}},\underset{i\in I}{\sup}$ $\nu_{A_{i}}).$  Then $A^{\ast^{'}}:=((\mu_{A})^{\ast^{'}},(\nu_{A})^{\ast^{'}})=\underset{i\in I}{\inf}$ $ (A_{i})^{\ast^{'}}=(\underset{i\in I}{\inf}$ $(\mu_{A_{i}})^{\ast^{'}},\underset{i\in I}{\sup}$ $(\nu_{A_{i}})^{\ast^{'}}).$
\end{lemma}

\begin{proof}
Let $[\alpha,x]\in R.$ Then $(\mu_{A})^{\ast^{'}}([\alpha,x])=(\underset{i\in I}{\inf}$ $\mu_{A_{i}})^{\ast^{'}}([\alpha,x])=\underset{s\in S}{\inf}$ $\underset{i\in I}{\inf}$ $\mu_{A_{i}}(s\alpha x)=\underset{i\in I}{\inf}$ $\underset{s\in S}{\inf}$ $\mu_{A_{i}}(s\alpha x)=\underset{i\in I}{\inf}$ $(\mu_{A_{i}})^{\ast^{'}}([\alpha,x]).$ Similarly we can show that $(\nu_{A})^{\ast^{'}}([\alpha,x])=\underset{i\in I}{\sup}$ $(\nu_{A_{i}})^{\ast^{'}}([\alpha,x]).$ Hence $A^{\ast^{'}}=\underset{i\in I}{\inf}$ $ (A_{i})^{\ast^{'}}.$
\end{proof}

\begin{proposition}
Let $S$ be a commutative $\Gamma$-semigroup with unities and $A=(\mu_{A},\nu_{A})$ be an $IFI(S).$ Then $<x,A>:=(<x,\mu_{A}>,<x,\nu_{A}>)$ is an $IFI(S)$ for all $x\in S.$
\end{proposition}

\begin{proof}
Let $R$ be the right operator semigroup of a $\Gamma$-semigroup $S.$ Since $S$ is commutative, $R$ is also commutative. By Proposition $2.8,$ $A^{\ast^{'}}=(\mu^{\ast^{'}}_{A},\nu^{\ast^{'}}_{A})$ is an $IFI(R).$ Let $x\in S.$ Then for any $\alpha\in\Gamma,$ $<[\alpha,x],A^{\ast^{'}}>=(<[\alpha,x],\mu_{A}^{\ast^{'}}>,<[\alpha,x],\nu_{A}^{\ast^{'}}>)$ is an $IFI(R)(cf.$ Proposition $3.2\cite{S4})$ and hence $(\underset{\alpha\in\Gamma}{\inf}$ $<[\alpha,x],\mu_{A}^{\ast^{'}}>,\underset{\alpha\in\Gamma}{\sup}$ $<[\alpha,x],\nu_{A}^{\ast^{'}}>)$ is an $IFI(R).$ So by Lemma $2.23(2),$ $(<x,\mu_{A}>^{\ast^{'}},<x,\nu_{A}>^{\ast^{'}})=<x,A>^{\ast^{'}}$ is an $IFI(R).$ Consequently, $(<x,A>^{\ast^{'}})^{\ast}$ is an $IFI(S)(cf.$ Proposition $2.7).$ Hence $<x,A>$ is an $IFI(S)(cf.$ Theorem $2.10).$
\end{proof}

Now by applying Proposition $2.17,$ Remark $5(i)$ we obtain the following proposition.

\begin{proposition}
Let $S$ be a commutative $\Gamma$-semigroup and $A=(\mu_{A},\nu_{A})$ be an $IFSPI(S).$ Then $<x,A>=(<x,\mu_{A}>,<x,\nu_{A}>)$ is an $IFSPI(S)$ for all $x\in S.$
\end{proposition}

\begin{proposition}
Let $S$ be a $\Gamma$-semigroup, $\{A_{i}\}_{i\in I}=(\mu_{A_{i}},\nu_{A_{i}})_{i\in I}$ be a non-empty family of $IFSPI(S)$ and let $A=(\mu_{A},\nu_{A})=\underset{i\in I}{\inf}$ $ A_{i}=(\underset{i\in I}{\inf}$ $\mu_{A_{i}},\underset{i\in I}{\sup}$ $\nu_{A_{i}}).$  Then for any $x\in S, <x,A>$ is an $IFSPI(S).$
\end{proposition}

\begin{proof}
Let $R$ be the right operator semigroup of $S.$ Since $S$ is commutative, $R$ is also commutative. Now $\{A_{i}^{\ast^{'}}\}_{i\in I}=(\mu_{A_{i}}^{\ast^{'}},\nu_{A_{i}}^{\ast^{'}})_{i\in I}$ is a non-empty family of $IFSPI(R)(cf.$ Proposition $2.17).$ Hence $\underset{i\in I}{\inf}$ $A_{i}^{\ast^{'}}$ is an $IFSPI(R).$ Thus by Lemma $2.27,$ $A^{\ast^{'}}$ is an $IFSPI(R).$ Then for any$[\alpha,x]\in R,$ $<[\alpha,x],A^{\ast^{'}}>=(<[\alpha,x],\mu_{A}^{\ast^{'}}>,<[\alpha,x],\nu_{A}^{\ast^{'}}>)$ is an $IFSPI(R)(cf.$ Proposition $2.11\cite{S4})$ and hence $(\underset{\alpha\in\Gamma}{\inf}$ $<[\alpha,x],\mu_{A}^{\ast^{'}}>,\underset{\alpha\in\Gamma}{\sup}$ $<[\alpha,x],\nu_{A}^{\ast^{'}}>)$ is an $IFSPI(R).$ So by Lemma $2.23(2),$ $(<x,\mu_{A}>^{\ast^{'}},<x,\nu_{A}>^{\ast^{'}})=<x,A>^{\ast^{'}}$ is an $IFSPI(R).$ Consequently, $(<x,A>^{\ast^{'}})^{\ast}$ is an $IFSPI(S)(cf.$ Proposition $2.16).$ Hence $<x,A>$ is an $IFSPI(S)(cf.$ Proposition $2.17$  and Remark $5(i)$)
\end{proof}

\begin{theorem}
Let $S$ be a commutative $\Gamma$-semigroup, $\{S_{i}\}_{i}$ be a non-empty family of $SPI(S),$  $A:=\underset{i\in I}{\bigcap} S_{i}\neq\phi$ and $M=(\mu_{A},\mu^{c}_{A})$ . Then $<x,M>$ is an $IFSPI(S)$ for all $x\in S.$
\end{theorem}

\begin{proof}
Since $\forall i\in I,$ $S_{i}$ is a $SPI(S),$ $S_{i}^{\ast^{'}}$ is a $SPI(R),$ $\forall i\in I.$ Now $A:=\underset{i\in I}{\bigcap} S_{i}.$ Then $A^{\ast^{'}}=(\underset{i\in I}{\bigcap} S_{i})^{\ast^{'}}=\underset{i\in I}{\bigcap} S_{i}^{\ast^{'}}(cf.$ Lemma $2.27)\neq\phi.$ So by Corollary $3.12\cite{S4},$ $<[\alpha,x],M>=(<[\alpha,x],\mu_{A}>,<[\alpha,x],\mu_{A}^{c}>)$ is an $IFSPI(R),$ $\forall\alpha\in\Gamma.$ This implies that $(\underset{\alpha\in\Gamma}{\inf}$ $<[\alpha,x],\mu_{A}>,\underset{\alpha\in\Gamma}{\sup}$ $<[\alpha,x],\mu_{A}^{c}>)$ is an $IFSPI(R),i.e.,(<[\alpha,x],\mu_{A}>^{\ast^{'}},<[\alpha,x],\mu_{A}^{c}>^{\ast^{'}})=<x,M>^{\ast^{'}}$ is an $IFSPI(R).$ Hence $(<x,M>^{\ast^{'}})^{\ast}$ is an $IFSPI(S).$ Consequently, by Theorem $2.18,$ $<x,M>$ is an $IFSPI(S).$
\end{proof}

To conclude the paper we obtain the following characterization of a prime ideal of a $\Gamma$-semigroup $S$ in terms of intutionistic fuzzy ideal extension.

\begin{theorem}
Let $S$ be a $\Gamma$-semigroup, $P$ be an $I(S)$ and $M=(\mu_{P},\mu^{c}_{P})$ where $\mu_{P}$ is the characteristic function of $P.$ Then $P$ is $PI(S)$ if and only if for $x\in S$ with $x\notin P,$ $<x,M>=M.$
\end{theorem}

\begin{proof}
Let $P$ be a $PI(S)$ and $x\notin P.$ Then $P^{\ast^{'}}$ is a $PI(R)(cf.$ Proposition $2.15).$ Also $[\alpha,x]\notin P^{\ast^{'}}$ for some $\alpha\in\Gamma.$ Then by Corollary $3.15\cite{S4},$ $<[\alpha,x],M^{\ast^{'}}>= M^{\ast^{'}} i.e., (<[\alpha,x],\mu^{\ast^{'}}_{P}>,<[\alpha,x],(\mu^{c}_{P})^{\ast^{'}}>=((\mu_{P})^{\ast^{'}},(\mu^{c}_{P})^{\ast^{'}})=M^{\ast^{'}}.$ Now by Lemma $2.12,$ $(<[\alpha,x],\mu^{\ast^{'}}_{P}>,<[\alpha,x],(\mu^{c}_{P})^{\ast^{'}}>)=(\mu_{P},\mu^{c}_{P})^{\ast^{'}}.$ Hence $(<[\alpha,x],M^{\ast^{'}}>)^{\ast}=(<[\alpha,x],\mu^{\ast^{'}}_{P}>,<[\alpha,x],(\mu^{c}_{P})^{\ast^{'}}>)^{\ast}=((\mu_{P},\mu^{c}_{P})^{\ast^{'}})^{\ast}=(\mu_{P},\mu^{c}_{P})=M.$ By Lemma $2.23(1),$ $<x,M>^{\ast^{'}}\subseteq <[\alpha,x],M^{\ast^{'}}>\forall\alpha\in\Gamma.$ So $(<x,M>^{\ast^{'}})^{\ast}\subseteq M.$ Consequently by Theorem $2.10,$ $<x,M>\subseteq M.$ Again by Proposition $5.5(1) \cite{S5},$ we have $M\subseteq <x,M>.$ Hence $<x,M>=M.$

Conversely, suppose $<z,M>=M$ for all $z\in S$ with$z\notin P.$ Let $x\Gamma y\subseteq P$ where $x,y\in S.$ Then $\mu_{P}(x\gamma y)=1$ and $\mu^{c}_{P}(x \gamma y)=0\forall\gamma\in\Gamma.$ Let $x\notin P.$ Then by hypothesis $<x,M>=M.$ This gives $<x,\mu_{P}>(y)=\mu_{P}(y)$ and $<x,\mu^{c}_{P}>(y)=\mu^{c}_{P}(y).$ Then $\underset{\gamma\in\Gamma}{\inf}$ $\mu_{P}(x\gamma y)=\mu_{P}(y)$ and $\underset{\gamma\in\Gamma}{\sup}$ $\mu^{c}_{P}(x\gamma y)=\mu^{c}_{P}(y)$ which implies that $\mu_{P}(y)=1$ and $\mu^{c}_{P}(y)=0$ whence $y\in P.$ Hence $P$ is a $PI(S).$
\end{proof}

%%%%%%%%%%%%%%%%%%%%%%%%%%%%%%%%%%%%%%%%%%%%%%%%%%%%%%%%%%%%%%%%%%%%%%%%%%%%%%%%%%
%%%%%%%%%%%%%%%%%%%%%%%%%%%%%%%%%%%%%%%%%%%%%%%%%%%%%%%%%%%%%%%%%%%%%%%%%%%%%%%%%%%%%%%%%

\end{document}